\documentclass[reqno, 11pt]{amsart}

\usepackage[text={155mm,235mm},centering]{geometry}
\input diagxy
\input xy

\usepackage[active]{srcltx} % SRC Specials

\newtheorem{thm}{Theorem}[section]
\newtheorem{cor}[thm]{Corollary}
\newtheorem{lem}[thm]{Lemma}
\newtheorem{prop}[thm]{Proposition}

\theoremstyle{definition}

\theoremstyle{property}

\theoremstyle{remark}
\newtheorem{rem}[thm]{Remark}

\numberwithin{equation}{section}
\usepackage[mathscr]{eucal}
\usepackage[all]{xy}
\usepackage{mathrsfs}
\usepackage{xypic}
\usepackage{amsfonts}
\usepackage{amsmath}
\usepackage{amsthm,bm}
\usepackage{amssymb,tikz-cd}
\usepackage{latexsym}
\usepackage{tabularx}
\usepackage{graphicx}
\usepackage{pict2e}
\usepackage[pagebackref]{hyperref}
\usepackage{tikz}
\usepackage{textcomp}
\usepackage[scr=rsfs]{mathalfa}
\definecolor{ceruleanblue}{rgb}{0.16, 0.32, 0.75}
\hypersetup{colorlinks=true,allcolors=ceruleanblue}
\allowdisplaybreaks

\begin{document}

\title[Holomorphic Koszul-Brylinski homology via Dolbeault cohomology]{Holomorphic Koszul-Brylinski homology via Dolbeault cohomology}

\author{Lingxu Meng}
\address{School of Mathematics, North University of China, Taiyuan, Shanxi 030051,  P. R. China}
\email{menglingxu@nuc.edu.cn}%


% -----------------------------------------------------------

\begin{abstract}
  We  use the Dolbeault cohomology  to investigate the Koszul-Brylinski homology  on  holomorphic Poisson manifolds.
  We obtain the Leray-Hirsch theorem for Hochschild homology and the Mayer-Vietoris sequence,  K\"{u}nneth theorem for holomorphic Koszul-Brylinski homology.
  In particular, we show some relations of  holomorphic Koszul-Brylinski homologies around  a blow-up transformation for the general case (\emph{not necessarily compact}) by our previous works on the Dolbeault cohomology.
  %Here, we use an explicit blow-up formula for Dolbeault cohomology given in our previous work, which can be induced by a morphism on the level of spaces of forms and %currents.
\end{abstract}

% -----------------------------------------------------------
\maketitle
% -----------------------------------------------------------

%==================================
\section{Introduction}
Holomorphic Poisson structures are a special class of Poisson structures, which naturally appear  in various fields \cite{EL1,EL2,Hue,I,KS}.
They have a close relationship with  generalized complex geometry.
In particular, the local model of  generalized complex manifolds is the product of a holomorphic Poisson manifold and a symplectic manifold  \cite{Bai}.
We refer the readers to \cite{Go,Gu,H1,H3,L-GSX,Py} and references therein for more results on  holomorphic Poisson structures.
%Bondal \cite{Bon} and Polishchuk \cite{P} first investigated holomorphic  Poisson structures from the point of view of algebraic geometry.

For  a holomorphic Poisson manifold $(X,\pi)$,
the Lichnerowicz-Poisson cohomology $H^k(X,\pi)$ and the Koszul-Brylinski homology $H_k(X,\pi)$  are two kinds of important invariants.
There are fruitful works on the former one \cite{CFP,CGP,Hong,HX,Ma,PS1,PS2}, etc., but few results on the later one  until now.
M. Sti\'{e}non \cite{S} used Lie algebroids to study holomorphic Koszul-Brylinski homology.
In particular, he proved that Evens-Lu-Weinstein pairing on  holomorphic Koszul-Brylinski homology is nondegenerate \cite[Theorem 4.4]{S}
 and  $H_k(X,\pi)\cong H^{n-k}(X,\pi)$ for the unimodular case \cite[Proposition 4.7]{S}.
He also obtained that the Euler characteristic of holomorphic Koszul-Brylinski homology   coincides with the signed Euler characteristic in the usual sense.
X. Chen, Y. Chen, S. Yang and X. Yang  gave a blow-up formula of  Koszul-Brylinski homology on compact holomorphic Poisson manifolds \cite[Theorem 1.1]{CCYY}
and computed the Koszul-Brylinski homology of del Pezzo surfaces and some complex parallelizable manifolds \cite[Section 6]{CCYY}.

Hochschild homology is a significant invariant of complex manifolds, which is widely studied in  noncommutative geometry and $K$-theory.
By the Hochschild-Kostant-Rosenberg theorem,
Hochschild homology is isomorphic to  the holomorphic Koszul-Brylinski homology for  the trivial holomorphic Poisson structure \cite{S,CCYY}.
Hence, the Hochschild homology can be investigated through the approach of holomorphic Koszul-Brylinski homology.
%\cite[Remark 4.2]{S}\cite{CCYY}.

The present paper aims to generalize several classical theorems in topology, such as Leray-Hirsch theorem, Mayer-Vietoris sequence, K\"{u}nneth theorem,  to  holomorphic Koszul-Brylinski homology.
In particular, we show some relations   of  Koszul-Brylinski homologies under the blow-up transformation of general (\emph{not necessarily compact})  holomorphic Poisson manifolds   by the self-intersection and the explicit expression of the blow-up formula of  Dolbeault cohomology in our previous works \cite{M1}.

%In Section 2, we recall some basic notions on complexes and give some properties of them, which may be well known for experts. In Section 3, the double complexes %$\mathcal{S}^{\bullet,\bullet}_{X}(\mathcal{L},s,t)$ and $\mathcal{T}^{\bullet,\bullet}_{X}(\mathcal{L},s,t)$ are defined and studied, which play important roles in the study %of the hypercohomology $\mathbb{H}^{k}(X,\Omega_X^{[s,t]}(\mathcal{L}))$.

\section{Preliminaries}
\subsection{Notations}
For a double complex $K^{\bullet,\bullet}$ (resp. a complex $K^\bullet$), $K^{\bullet,\bullet}[m,n]$ (resp.  $K^\bullet[m]$) means the shifted double complex (resp. shifted complex) by bidegree $(m,n)$ (resp.  degree $m$), where $m$, $n\in\mathbb{Z}$.
Denote by $sK^{\bullet,\bullet}$ the (simple) complex associated to the double complex $K^{\bullet,\bullet}$
and by $ss(K^{\bullet,\bullet}\otimes L^{\bullet,\bullet})$ the double complex associated to the tensor of double complexes $K^{\bullet,\bullet}$ and $L^{\bullet,\bullet}$.
See \cite[Sections 2.1, 2.2, 2.4]{M2} for more details.

For a complex manifold $X$, denote by  $H^{p,q}(X)$ the Dolbeault cohomology of $X$ and by $\mathcal{A}_X^{k}$ (resp.  $\mathcal{D}_X^{k}$, $\mathcal{A}_X^{p,q}$, $\mathcal{D}_X^{p,q}$, $\mathcal{O}_X$, $\Omega_X^p$) the sheaf of germs of complex-valued smooth $k$-forms (resp.  complex-valued $k$-currents, smooth $(p,q)$-forms, $(p,q)$-currents, holomorphic functions, holomorphic $p$-forms) on $X$ for any $p$, $q$, $k\in\mathbb{Z}$.
Notice that, if $\textrm{dim}_{\mathbb{C}}X=n$, we tacitly approve that $H^{p,q}(X)=0$, $\mathcal{A}_X^{p,q}=0$, $\mathcal{D}_X^{p,q}=0$ for $p<0$ or $>n$, or $q<0$ or $>n$.

For complex vector spaces $U_1$, $\cdots$, $U_n$,
the vector in $\bigoplus_{i=1}^nU_i$ is written as $(u_1,\,\cdots,\,u_n)^T$, where $u_i\in U_i$ for $1\leq i\leq n$ and $(\bullet)^T$ means the transposition.
For  complex linear maps $f_{ji}:U_i\rightarrow V_j$,  $1\leq i\leq n$, $1\leq j\leq m$,
define
\begin{displaymath}
\left(
  \begin{array}{ccc}
   f_{11}  &  \cdots &   f_{1n}       \\
   \cdots   & \cdots &     \cdots     \\
   f_{m1}  &  \cdots &   f_{mn}       \\
  \end{array}
\right)
:\bigoplus_{i=1}^nU_i\rightarrow\bigoplus_{j=1}^mV_i
\end{displaymath}
as
\begin{displaymath}
(u_1,\,\cdots,\,u_n)^T\mapsto \left(\sum_{i=1}^nf_{1i}(u_i),\,\cdots,\,\sum_{i=1}^nf_{mi}(u_i)\right)^T.
\end{displaymath}
%\begin{displaymath}
%\left(
%  \begin{array}{ccc}
%   f_{11}  &  \cdots &   f_{1n}       \\
%   \cdots   & \cdots &     \cdots     \\
%   f_{m1}  &  \cdots &   f_{mn}       \\
%  \end{array}
%\right)
%\left(
%  \begin{array}{c}
%   u_1         \\
%   \cdots    \\
%   u_n         \\
%  \end{array}
%\right)
%=
%\left(
%  \begin{array}{c}
%   \sum_{i=1}^nf_{1i}(u_i)         \\
%   \cdots    \\
%    \sum_{i=1}^nf_{mi}(u_i)        \\
%  \end{array}
%\right)
%\end{displaymath}
In particular, $(f,g):U\oplus V\rightarrow W$ means $(u,v)^T\mapsto f(u)+g(v)$ for $f:U\rightarrow W$ and $g:V\rightarrow W$
and $(f,g)^T:W\rightarrow U\oplus V$ means $w\mapsto (f(w),\,g(w))^T$ for  $f:W\rightarrow U$ and $g:W\rightarrow V$.

\subsection{Holomorphic Poisson manifolds}
Suppose that $X$ is a complex manifold and $\mathcal{O}_X(U)$ is endowed with a Poisson bracket $\{-,-\}$ for any open subset $U\subseteq X$ satisfying that the restriction map $\mathcal{O}_X(U)\rightarrow \mathcal{O}_X(V)$ is a morphism of Poisson algebras for open sets $V\subseteq U$.
Then $(X,\{-,-\})$ is said to be a \emph{holomorphic Poisson manifold}.
In this case, there exists a unique holomorphic bivector field $\pi\in H^0(X,\bigwedge^2T_X)$ such that  $\pi(df\wedge dg)=\{f,g\}$ for any $f$, $g\in \mathcal{O}_X(U)$, where $T_X$ is the holomorphic tangent bundle of $X$.
We also write $(X,\pi)$ for this holomorphic Poisson manifold and say that $\pi$ is a \emph{holomorphic Poisson structure} on $X$.

A holomorphic map $\rho:X\rightarrow Y$ between holomorphic Poisson manifolds $(X,\{-,-\})$ and $(Y,\{-,-\})$ is called a \emph{morphism} if the pullback $\rho^*:\mathcal{O}_Y(U)\rightarrow \mathcal{O}_X(\rho^{-1}(U))$ is a morphism of Poisson algebras for each open set $U\subseteq Y$.
A holomorphic map $\rho:X\rightarrow Y$ between holomorphic Poisson manifolds $(X,\pi)$ and $(Y,\sigma)$ is a  morphism
if and only if $\rho_*(\pi_x)=\sigma_{\rho(x)}$ for all $x\in X$.

Let $(X,\pi)$, $(Y,\sigma)$ be holomorphic Poisson manifolds and let $Y$ be also a closed complex submanifold of $X$.
We say $(Y,\sigma)$ is a \emph{closed holomorphic Poisson submanifold} of $(X,\pi)$, if the natural inclusion $i:(Y,\sigma)\rightarrow(X,\pi)$ is a morphism,
i.e., $i_*(\sigma_y)=\pi_{y}$ for all $y\in Y$.
%For a holomorphic Poisson manifold $(X,\pi)$ and a complex submanifold $Y\subseteq X$,
Evidently, there exists at most one holomorphic Poisson structure $\sigma$ on $Y$ such that $(Y,\sigma)$ is a closed holomorphic Poisson submanifold of $(X,\pi)$.
In such case,  denote $\sigma$ by $\pi|_Y$.

%For a holomorphic Poisson manifold $(X,\{-,-\})$, there exists a unique holomorphic bivector field $\pi\in \Gamma(X,\bigwedge^2\mathcal{T}_X)$ such that  $\pi(df\wedge %dg)=\{f,g\}$ for any $f$, $g\in \mathcal{O}_X(U)$, where $\mathcal{T}_X$ is the holomorphic tangent bundle of $X$.
%We also write $(X,\pi)$ as this holomorphic Poisson manifold.
%By the definitions, a holomorphic map $\rho:X\rightarrow Y$ between holomorphic Poisson manifolds $(X,\pi)$ and $(Y,\sigma)$ is a Poisson morphism
%if and only if $\rho_*(\pi_x)=\sigma_{\rho(x)}$ for all $x\in X$.

Let $(X,\pi)$ be a holomorphic Poisson manifold.
Denote by $l_{\pi}:\mathcal{A}^k(X)\rightarrow \mathcal{A}^{k-2}(X)$ the contraction by $\pi$ and set $\partial_{\pi}:=l_{\pi}\circ\partial-\partial\circ l_{\pi}:\mathcal{A}_X^k\rightarrow \mathcal{A}_X^{k-1}$ for any $k\in\mathbb{Z}$.
Then $\partial_{\pi}(\Omega_X^{p})\subseteq \Omega_X^{p-1}$, $\partial_{\pi}(\mathcal{A}_X^{p,q})\subseteq \mathcal{A}_X^{p-1,q}$, $\partial_{\pi}^2=0$ and $\bar{\partial}\partial_{\pi}+\partial_{\pi}\bar{\partial}=0$.
If $\rho:(X,\pi)\rightarrow (Y,\sigma)$ is a morphism,
then  $\partial_\pi\rho^*=\rho^*\partial_\sigma$.

\subsection{Holomorphic Koszul-Brylinski homology}
In this subsection, we collect some element knowledge on holomorphic Koszul-Brylinski homology; see also \cite[Section 4]{S} and \cite[Section 2.2, 3.1]{CCYY}.
Let $(X,\pi)$ be a holomorphic Poisson manifold with complex dimension $n$.
\subsubsection{Definition}
The holomorphic Koszul-Brylinski complex $\mathcal{M}_X^\bullet(\pi)$ of $X$ is the complex of sheaves
\begin{displaymath}
\xymatrix{
0\ar[r] &\Omega_X^n\ar[r]^{\partial_\pi\,} &\Omega_X^{n-1}\ar[r]^{\,\,\partial_\pi}&\cdots\ar[r]^{\partial_\pi\,}&\Omega_X^1\ar[r]^{\,\partial_\pi}&\mathcal{O}_X\ar[r]&0.
}
\end{displaymath}
More precisely, $\mathcal{M}_X^k(\pi)=\Omega_X^{n-k}$, $d^k=\partial_\pi$.
The $k$-th  \emph{Koszul-Brylinski homology} of $(X,\pi)$ is defined as $H_k(X,\pi):=\mathbb{H}^k(X,\mathcal{M}_X^\bullet(\pi))$.

\subsubsection{Computation via smooth forms}
Set $\mathcal{K}_X^{p,q}(\pi):=\mathcal{A}_X^{-p,q}$, $d^{p,q}_1:=\partial_\pi$,  $d^{p,q}_2:=\bar{\partial}$.
Then $(\mathcal{K}_X^{\bullet,\bullet}(\pi),d_1,d_2)$ is a double complex, shortly denoted  by $\mathcal{K}^{\bullet,\bullet}_{X}(\pi)$.
Let $\mathcal{K}^{\bullet}_{X}(\pi):=s\mathcal{K}^{\bullet,\bullet}_{X}(\pi)$  be the complex associated to $\mathcal{K}^{\bullet,\bullet}_{X}(\pi)$.
Set $K^{p,q}(X,\pi):=\Gamma(X,\mathcal{K}^{p,q}_{X}(\pi))$ and $K^{p}(X,\pi):=\Gamma(X,\mathcal{K}^{p}_{X}(\pi))$.
For any $p\in\mathbb{Z}$, $\mathcal{M}_X^{p+n}(\pi)\rightarrow (\mathcal{K}^{p,\bullet}_{X}(\pi),d^{p,\bullet}_2)$ given by the inclusion is a resolution of $\mathcal{M}_X^{p+n}(\pi)$.
By \cite[Lemma 8.5]{V}, the inclusion gives a quasi-isomorphism $\mathcal{M}_X^{\bullet}(\pi)[n]\rightarrow \mathcal{K}^{\bullet}_{X}(\pi)$ of complexes of sheaves.
For any $p\in\mathbb{Z}$, $\mathcal{K}^{p}_{X}(\pi)$ is a soft sheaf, so it is $\Gamma$-acyclic.
By \cite[Proposition 8.12]{V},
\begin{equation}\label{computation1}
H_k(X,\pi)\cong H^{k-n}(K^{\bullet}(X,\pi))
\end{equation}
for any $k\in\mathbb{Z}$.
%See also \cite[Theorem 5.1]{S} and \cite[Proposition 2.5]{CCYY}.
Associated to $K^{\bullet,\bullet}(X,\pi)$, there is a spectral sequence $_KE_r^{p,q}(X,\pi)\Rightarrow H_{p+q+n}(X,\pi)$,
%\begin{equation}\label{spectral sequence}
%_KE_1^{p,q}(X,\pi)\Rightarrow H_{p+q+n}(X,\pi),
%\end{equation}
where
\begin{equation}\label{ss1}
_KE_1^{p,q}(X,\pi)=H^q(K^{p,\bullet}(X,\pi))=H^{-p,q}(X).
\end{equation}

\subsubsection{Computation via currents}
For $T\in \mathcal{D}_X^{k}$, $\alpha\in \mathcal{A}_X^{2n-k+1}$, $\beta\in \mathcal{A}_X^{2n-k-1}$,
set $(\partial_{\pi}T)(\alpha):=(-1)^{k-1}T(\partial_{\pi}\alpha)$ and $(\bar{\partial} T)(\beta):=(-1)^{k+1}T(\bar{\partial}\beta)$.
Clearly, $\partial_{\pi}(\mathcal{D}_X^{p,q})\subseteq \mathcal{D}_X^{p-1,q}$, $\partial_{\pi}^2=0$ and $\bar{\partial}\partial_{\pi}+\partial_{\pi}\bar{\partial}=0$.
Let $\rho:(X,\pi)\rightarrow (Y,\sigma)$ be a  morphism of holomorphic Poisson manifolds and denote  by $\rho_*$ the pushforward of currents.
Then  $\partial_\sigma\rho_*=\rho_*\partial_\pi$.

As those of $\mathcal{K}(\pi)$, we can define $\mathcal{P}^{\bullet,\bullet}_{X}(\pi)$, $\mathcal{P}^{\bullet}_{X}(\pi)$, $P^{\bullet,\bullet}(X,\pi)$, $P^{\bullet}(X,\pi)$ and $_PE_1^{p,q}(X,\pi)$,
where  $\mathcal{P}_X^{p,q}(\pi):=\mathcal{D}_X^{-p,q}$, $d^{p,q}_1:=\partial_\pi$,  $d^{p,q}_2:=\bar{\partial}$.
Similarly, the inclusion gives a quasi-isomorphism $\mathcal{M}_X^{\bullet}(\pi)[n]\rightarrow \mathcal{P}^{\bullet}_{X}(\pi)$ of complexes of sheaves.

The inclusion naturally gives the morphism $K^{\bullet,\bullet}(X,\pi)\hookrightarrow P^{\bullet,\bullet}(X,\pi)$ of double complexes.
It induces isomorphisms $_KE_1^{p,q}(X,\pi)=H^q(\mathcal{A}^{-p,\bullet}(X))\tilde{\rightarrow} _PE_1^{p,q}(X,\pi)=H^q(\mathcal{D}^{-p,\bullet}(X))$ for all $p$, $q\in \mathbb{Z}$,
hence induces  isomorphisms $H^k(K^\bullet(X,\pi))\rightarrow H^k(P^\bullet(X,\pi))$ for all $k\in \mathbb{Z}$.
Both $H^k(K^\bullet(X,\pi))$ and $H^k(P^\bullet(X,\pi))$ will be written as $H_k(X,\pi)$.

\subsubsection{Pullback and Pushforward}
Let $\rho:(X,\pi)\rightarrow (Y,\sigma)$ be a   morphism of holomorphic Poisson manifolds and set $r=\textrm{dim}_{\mathbb{C}}X-\textrm{dim}_{\mathbb{C}}Y$.
Then  $\rho$ induces the \emph{pullback}
$\rho^*:K^{\bullet,\bullet}(Y,\sigma)\rightarrow K^{\bullet,\bullet}(X,\pi)$, hence induces
$\rho^*:H_{k}(Y,\sigma)\rightarrow H_{k+r}(X,\pi)$.
In addition, if $\rho$ is proper, it induces the \emph{pushforward} $\rho_*:P^{\bullet,\bullet}(X,\pi)\rightarrow P^{\bullet,\bullet}(Y,\sigma)[r,-r]$,
hence induces  $\rho_*:H_{k}(X,\pi)\rightarrow H_{k-r}(Y,\sigma)$.
%$\rho^*:H^{k-m}(K^\bullet(Y,\sigma))\rightarrow H^{k-m}(K^\bullet(X,\pi))$ and  $\rho_*:H^{k-n}(P^\bullet(X,\pi))\rightarrow H^{k-n}(P^\bullet(Y,\sigma))$.

\section{Hochschild homology}
Assume that $X$ is a complex manifold with complex dimension $n$.
Denote by $\Delta: X \rightarrow X \times X$ the diagonal embedding.
Its image is  a complex submanifold of $X\times X$ isomorphic to $X$, still denoted by $\Delta$.
The Hochschild homology of $X$  is defined as
$\textrm{HH}_k(X):= \textrm{Tor}_k^{X\times X} (\mathcal{O}_{\Delta}, \Delta_{*}\mathcal{O}_X)$ for each $k\in \mathbb{Z}$. % for $-n\leq k\leq n$.
%$\textrm{HH}_k(X):= \textrm{Hom}_ {\textrm{D}^b_{coh}(X\times X)} (\Delta_{!}\mathcal{O}_X[k],\mathcal{O}_{\Delta})$,
%where $\Delta_{!}$ is the left adjoint of the pullback functor  (cf. \cite[Section 1.1]{Ca}).
%where $\mathcal{O}_{\Delta}$ is the structure sheaf of the diagonal embedding $\Delta: X \rightarrow X \times X$, and $\Delta_{!}$ is the left adjoint of the pullback functor  %(cf. \cite[Section 1.1]{Ca}).
By the Hochschild-Kostant-Rosenberg theorem \cite[Corollary 3.1.4]{BF}, %\cite[Corollary 4.2]{Ca},
%$\textrm{HH}_{k}(X)\cong\bigoplus\limits_{p-q=k}H^{q}(X,\Omega^q_X)$.
\begin{equation}\label{HKR}
\textrm{HH}_{k}(X)\cong\bigoplus\limits_{p-q=k}H^{q}(X,\Omega^p_X).
\end{equation}
Set  $\pi=0$.
In such case, $(X,\pi)$ is a holomorphic Poisson manifold and $\partial_\pi=0$.
Moreover,
\begin{equation}\label{zero1}
\mathcal{M}_X^\bullet(0)=\bigoplus\limits_{i=0}^n\Omega_X^{n-i}[-i],
\end{equation}
where $\Omega_X^i$ means the complex of sheaves concentrated on $0$-th term with the sheaf $\Omega_X^i$.
By (\ref{zero1}), we have
\begin{equation}\label{zero2}
H_k(X,0)=\bigoplus\limits_{p-q=n-k}H^{q}(X,\Omega^p_X).
\end{equation}
Combining (\ref{HKR}) and (\ref{zero2}), we have
\begin{equation}\label{zero3}
H_k(X,0)\cong\textrm{HH}_{n-k}(X).
\end{equation}
See also \cite[Remark 4.2]{S}\cite[p.17]{CCYY} for other discussions.
Hence, we can study the Hochschild homology via the holomorphic Koszul-Brylinski homology.

\begin{thm}\label{L-H}
Let $\rho:E\rightarrow X$ be a holomorphic fiber bundle  over a complex manifold $X$.
Assume that there exist  $d$-closed  forms $t_1$, $\ldots$, $t_r$ of pure degrees on $E$ such that the restrictions  of their Dolbeault  classes $[t_1]_{\bar{\partial}}$, $\ldots$, $[t_r]_{\bar{\partial}}$ to $E_x$  is a basis of $H^{\bullet,\bullet}(E_x)=\bigoplus_{p,q\geq 0}H^{p,q}(E_x)$ for every $x\in X$.
Then there exists an isomorphism
%\begin{displaymath}
%\bigoplus\limits_{i=1}^rH_{k+u_i-v_i}(X,0,\mathcal{L}) \tilde{\rightarrow} H_{k+m}(E,0,\rho^{-1}\mathcal{L})
%\end{displaymath}
\begin{displaymath}
\bigoplus\limits_{i=1}^r\emph{HH}_{k+v_i-u_i}(X) \tilde{\rightarrow} \emph{HH}_{k}(E)
\end{displaymath}
for any $k\in\mathbb{Z}$, where $(u_i,v_i)$ is the bidegree of $t_i$ for $1\leq i\leq r$.
\end{thm}
\begin{proof}
Set $n=\textrm{dim}_{\mathbb{C}}X$ and $m=\textrm{dim}_{\mathbb{C}}E-\textrm{dim}_{\mathbb{C}}X$.
Set
\begin{displaymath}
S^{\bullet,\bullet}:=\bigoplus\limits_{i=1}^{r}K^{\bullet,\bullet}(X,0)[u_i,-v_i] \quad\mbox{ and }\quad T^{\bullet,\bullet}:=K^{\bullet,\bullet}(E,0).
\end{displaymath}
By (\ref{ss1}), we get the first pages
\begin{displaymath}
_{S}E_1^{p,q}=\bigoplus\limits_{i=1}^{r}H^{-(p+u_i),q-v_i}(X) \quad\mbox{ and }\quad _{T}E_1^{p,q}=H^{-p,q}(E)
\end{displaymath}
of the spectral sequences associated to $S^{\bullet,\bullet}$ and $T^{\bullet,\bullet}$ respectively.
By \cite[Theorem 4.2]{M3}, the morphism $(\rho^*(\bullet)\wedge t_1,\,\cdots,\,\rho^*(\bullet)\wedge t_r):S^{\bullet,\bullet}\rightarrow T^{\bullet,\bullet}$
of double complexes induces an isomorphism  $_SE_1^{\bullet,\bullet}\rightarrow\mbox{ }  _TE_1^{\bullet,\bullet}$  at $E_1$-pages, hence induces an isomorphism $H^{-k}(sS^{\bullet,\bullet})\rightarrow H^{-k}(sT^{\bullet,\bullet})$ for any $k\in\mathbb{Z}$.
Notice that $sS^{\bullet,\bullet}= \bigoplus_{i=1}^{r}K^{\bullet}(X,0)[u_i-v_i]$ and $sT^{\bullet,\bullet}=K^{\bullet}(E,0)$.
By (\ref{computation1}) and (\ref{zero3}),
\begin{displaymath}
H^{-k}(sS^{\bullet,\bullet})\cong\bigoplus_{i=1}^{r}H_{-k+u_i-v_i+n}(X,0)\cong \bigoplus_{i=1}^{r}\textrm{HH}_{k+v_i-u_i}(X)
\end{displaymath}
%\begin{displaymath}
%\begin{aligned}
%H^{k}(sS^{\bullet,\bullet})=&\bigoplus_{i=1}^{r}H^{k+u_i-v_i}(K^{\bullet}(X,0))=\bigoplus_{i=1}^{r}H_{k+u_i-v_i+n}(X,0)\\
%\cong& \bigoplus_{i=1}^{r}\textrm{HH}_{v_i-u_i-k}(X)
%\end{aligned}
%\end{displaymath}
and
\begin{displaymath}
H^{-k}(sT^{\bullet,\bullet})\cong H_{-k+n+m}(E,0)\cong\textrm{HH}_{k}(E).
\end{displaymath}
We complete the proof.
\end{proof}

\begin{cor}\label{flagbun}
Let $E$ be the flag bundle associated to a holomorphic vector bundle over a complex manifold $X$.
Denote by $F$ the general fiber of $E$ over $X$.
Then there exists an isomorphism
\begin{displaymath}
\emph{HH}_{k}(X)^{\oplus b(F)} \tilde{\rightarrow} \emph{HH}_{k}(E)
\end{displaymath}
for any $k\in\mathbb{Z}$, where $b(F)$ is the sum of all betti numbers of $F$.
\end{cor}
\begin{proof}
Assume that $V$ is a holomorphic vector bundle with rank $n$ over $X$ and  $(n_1, \ldots, n_r)$ is a sequence of positive integers with $\sum_{i=1}^rn_i=n$, such that the fiber $E_x$ of $E$ over $x\in X$ is the flag manifold $\textrm{Fl}(n_1,\ldots,n_r)(V_x)=$  %$F(n_1,\ldots,n_r)(E_x)=\{(F_1,\ldots,F_{r-1})|0=F_0\subsetneq F_1\subsetneq F_2\subsetneq \ldots \subsetneq F_{r-1}\subsetneq F_r=E_x, \mbox{ where } F_i \mbox{ is a complex vector space with }\textrm{dim}_{\mathbb{C}}F_i=\sum\limits_{j=1}^{i}n_j \mbox{  for every } i\}$.
\begin{displaymath}
\begin{aligned}
\{\,(W_0,W_1,\ldots,W_{r-1},W_r)\,|\,&0=W_0\subsetneq W_1\subsetneq W_2\subsetneq \ldots \subsetneq W_{r-1}\subsetneq W_r=V_x, \mbox{ where } W_i \mbox{ is a } \\
& \mbox{ complex vector space with dimension } \sum_{j=1}^{i}n_j \mbox{  for } 1\leq i\leq r\,\}.
\end{aligned}
\end{displaymath}
For $0\leq i\leq r$, assume that $V_i$ is the  universal subbundle over $E$ whose fiber over the point
$(W_0,W_1,\ldots,W_{r-1},W_r)$ is $W_i$.
Notice that $V_0=0$ and $V_r=\rho^*V$, where $\rho$ is the projection from $E$ onto $X$.
For $1\leq i\leq r$, set  $V^{(i)}=V_i/V_{i-1}$  and denote by $t_j^{(i)}\in A^{j,j}(E)$ a $j$-th Chern form of $V^{(i)}$.
For any $x\in X$, the restrictions $t_{j_i}^{(i)}|_{E_x}$ ($1\leq i\leq r$, $1\leq j_i\leq n_i$) to $E_x$ are  Chern forms of successive  universal quotient bundles of the flag manifold $E_x$.
As we know, there exists the monomials $P_i(T_1^{(1)}, \ldots,  T_{n_1}^{(1)}, \ldots, T_1^{(r)}, \ldots, T_{n_r}^{(r)})$ for $1\leq i\leq l$ such that
\begin{displaymath}
t_i:=P_i([t_{1}^{(1)}|_{E_x}]_{\bar{\partial}}, \ldots,  [t_{n_1}^{(1)}|_{E_x}]_{\bar{\partial}}, \ldots, [t_1^{(r)}|_{E_x}]_{\bar{\partial}}, \ldots, [t_{n_r}^{(r)}|_{E_x}]_{\bar{\partial}}), \mbox{  } 1\leq i\leq l,
\end{displaymath}
is a basis of  $H^{\bullet,\bullet}(E_x)$.
Clearly,  all $t_i$ are $d$-closed on $E$ and $l=\textrm{dim}_{\mathbb{C}}H^{\bullet,\bullet}(E_x)$.
By the Hodge decomposition theorem, $l=b(F)$, since the flag manifold $F=E_x$ is  projective.
By Theorem \ref{L-H}, the corollary follows.
\end{proof}

\begin{rem}
A flag manifold can be viewed as a flag bundle over a single point.
Its Hochschild homology  can be obtained by Corollary \ref{flagbun}, which is a special case of the following Section 3.2.1.
\end{rem}

%\begin{cor}\label{projbun}
%Let $\rho:E\rightarrow X$ be the projective bundle of a holomorphic vector bundle of rank $r$   over a complex manifold $X$.
%Then there exists an isomorphism
%\begin{displaymath}
%\emph{HH}_{k}(X)^{\oplus r} \tilde{\rightarrow} \emph{HH}_{k}(E)
%\end{displaymath}
%for any $k$.
%\end{cor}

%\begin{rem}
%Essentially, the projective bundle formula of Hochschild homology (Corollary \ref{projbun}) was first obtained  in \cite[Theorem 4.1]{CCYY}.
%\end{rem}

\section{Holomorphic Koszul-Brylinski homology}
\subsection{Stein manifolds and flag manifolds}
Let $(X,\pi)$ be a holomorphic Poisson manifold of complex dimension $n$.

Assume that $X$ is a Stein manifold.
Then
\begin{displaymath}
_KE_1^{p,0}(X,\pi)=\Gamma(X,\Omega_X^{-p}), \quad _KE_1^{p,q}(X,\pi)=0 \mbox{ for } q\neq0,
\end{displaymath}
and
$d^{p,0}_1=\partial_{\pi}$,  $d^{p,q}_1=0$ for $q\neq0$.
%\begin{displaymath}
%d_1=\partial_{\pi}:_KE_1^{p,0}(X,\pi)\rightarrow _KE_1^{p+1,0}(X,\pi),
%\end{displaymath}
%\begin{displaymath}
%d_1=0:_KE_1^{p,q}(X,\pi)\rightarrow _KE_1^{p+1,q}(X,\pi) \mbox{ for } q\neq0,
%\end{displaymath}
We obtained that $_KE_2^{p,0}(X,\pi)=H^{-p}(\Gamma(X,\Omega_X^{\bullet}),\partial_{\pi})$ and $E_2^{p,q}(X,\pi)=0$ for $q\neq0$.
Hence the spectral sequence $_KE_1^{\bullet,\bullet}(X,\pi)$ degenerates at $E_2$-pages and
\begin{displaymath}
H_{k}(X,\pi)=H^{-k+n}(\Gamma(X,\Omega_X^{\bullet}),\partial_{\pi})=H^k(\Gamma(X,\mathcal{M}_X^{\bullet}(\pi)))
\end{displaymath}
for any $k\in\mathbb{Z}$.
%\begin{displaymath}
%H_{k}(X,\pi)=H^{k-n}(K^{\bullet}(X,\pi))=H^{-k+n}(\Gamma(X,\Omega_X^{\bullet}),\partial_{\pi})=H^k(\Gamma(X,\mathcal{M}_X^{\bullet}))
%\end{displaymath}

%So $H_{k}(X,\pi)=H^k(\Gamma(X,\mathcal{M}_X^{\bullet}))$.

Assume that $X$ is a flag manifold.
Then $_KE_1^{p,q}(X,\pi)=H^{-p,q}(X)=0$ for $p+q\neq 0$ and hence all $d_1^{p,q}=0$.
The spectral sequence $_KE_r^{\bullet,\bullet}(X,\pi)$ degenerates at $E_1$-pages.
We have
\begin{equation}\label{flagmanifold}
H_{k}(X,\pi)=\bigoplus_{p+q=k-n} H^{-p,q}(X)=\left\{
 \begin{array}{ll}
\mathbb{C}^{b(X)},&~k= n\\
 &\\
 0,&~\textrm{others}
 \end{array}
 \right.
\end{equation}
for any $k\in\mathbb{Z}$, where $b(X)$ is the sum of all betti numbers of $X$.

\subsection{Mayer-Vietoris sequence}
Suppose that $(X,\pi)$ is a holomorphic Poisson manifold.
For each open subset $U\subseteq X$, $(U,\pi|_U)$ is naturally a holomorphic Poisson manifold.
Shortly write $H_k(U,\pi|_U)$  as $H_{k}(U,\pi)$.
The inclusion $j:U\hookrightarrow X$ induces the pullback $j^*:H_{k}(X,\pi)\rightarrow H_{k}(U,\pi)$.
We have the Mayer-Vietoris type sequence for holomorphic Koszul-Brylinski homology as follows.
\begin{prop}\label{M-V}
Suppose that $(X,\pi)$ is a holomorphic Poisson manifold.   For open subsets  $U$, $V\subseteq X$, there is a long exact sequence
\begin{displaymath}
\small{
\xymatrix{
\cdots\ar[r]^{} &H_{k}(U\cup V,\pi)\ar[r]^{(j_1^*,j_2^*)^T\,\,\,\quad} &
H_{k}(U,\pi)\oplus H_{k}(V,\pi)\ar[r]^{\,\,\,\quad (j_1^*,-j_2^*)}&
H_{k}(U\cap V,\pi)\ar[r]^{\quad g}&
H_{k+1}(U\cup V,\pi)\ar[r]^{} & \cdots
},}
\end{displaymath}
%\begin{displaymath}
%\cdots\rightarrow H_{k}(U\cup V,\pi)\rightarrow H_{k}(U,\pi)\oplus H_{k}(V,\pi)\rightarrow H_{k}(U\cap V,\pi)\rightarrow H_{k+1}(U\cup V,\pi)\rightarrow\cdots,
%\end{displaymath}
where $j_1:U\cap V\rightarrow U$ and  $j_2:U\cap V\rightarrow V$ are inclusions.

%$(2)$ Let $f:(X,\pi)\rightarrow (Y,\sigma)$ is a morphism of holomorphic Poisson mani
\end{prop}
\begin{proof}
Denote $f=(j_1^*,j_2^*)^T$ and $g=(j_1^*,-j_2^*)$.
For any $p$, $q\in\mathbb{Z}$, consider the sequence
\begin{equation}\label{exact}
\xymatrix{
0\ar[r]^{} &\mathcal{A}^{-p,q}(U\cup V)\ar[r]^{f\quad} &
\mathcal{A}^{-p,q}(U)\oplus \mathcal{A}^{-p,q}(V)\ar[r]^{\quad g}&
\mathcal{A}^{-p,q}(U\cap V)\ar[r]^{} & 0
}.
\end{equation}
Evidently, $f$ is injective and $\textrm{ker}g=\textrm{im} f$.
Let $\{\rho_U, \rho_V\}$ be a partition of unity subordinate to the open covering $\{U,V\}$ of $U\cup V$.
That is to say, $\rho_U$, $\rho_V\in C^{\infty}(U\cup V)$ satisfy that $\rho_U+\rho_V=1$ and $\textrm{supp}\rho_U\subseteq U$, $\textrm{supp}\rho_V\subseteq V$.
For any $\eta\in \mathcal{A}^{-p,q}(U\cap V)$, $\rho_V\eta\in \mathcal{A}^{-p,q}(U)$ and $-\rho_U\eta \in\mathcal{A}^{-p,q}(V)$.
Clearly, $g(\rho_V\eta,-\rho_U\eta)^T=\eta$.
So $g$ is surjective.
We proved that (\ref{exact}) is exact.
Hence, we easily obtain the short exact sequence  of complexes
\begin{displaymath}
0\rightarrow K^{\bullet}(U\cup V,\pi)\rightarrow K^{\bullet}(U,\pi)\oplus K^{\bullet}(V,\pi)\rightarrow K^{\bullet}(U\cap V,\pi)\rightarrow 0,
\end{displaymath}
which induces the long exact sequence in this proposition.
\end{proof}

%\subsection{K\"{u}nneth theorem}
%Let $X$, $Y$ be complex manifolds and  let $pr_1$, $pr_2$ be projections from $X\times Y$ onto $X$, $Y$, respectively.
%For coherent analytic sheaves $\mathcal{F}$ and $\mathcal{G}$ of $\mathcal{O}_X$- and $\mathcal{O}_Y$-modules  on $X$ and $Y$ respectively,  the \emph{analytic external tensor %product} of $\mathcal{F}$ and $\mathcal{G}$ on $X\times Y$  is defined as
%\begin{displaymath}
%\mathcal{F}\boxtimes\mathcal{G}=pr_1^*\mathcal{F}\otimes_{\mathcal{O}_{X\times Y}}pr_2^*\mathcal{G}.
%\end{displaymath}

\subsection{K\"{u}nneth theorem}
Suppose that $(X,\pi)$ and $(Y,\sigma)$ are holomorphic Poisson manifolds.
Define the bivector filed $\pi\oplus\sigma$ as follows:

Let $(U,z_1, \ldots, z_n)$ and $(V,w_1, \ldots, w_m)$ be the charts  of local coordinates of $X$ and $Y$ respectively.
If $\pi|_U=\sum_{1\leq i,j\leq n} \pi_{ij}(z)\frac{\partial}{\partial z_i}\wedge \frac{\partial}{\partial z_j}$
and $\sigma|_U=\sum_{1\leq k,l\leq m} \sigma_{kl}(w)\frac{\partial}{\partial w_k}\wedge \frac{\partial}{\partial w_l}$,
then $\pi\oplus \sigma$ is defined as
\begin{displaymath}
\sum_{1\leq i,j\leq n} \pi_{ij}(z)\frac{\partial}{\partial z_i}\wedge \frac{\partial}{\partial z_j}+\sum_{1\leq k,l\leq m} \sigma_{kl}(w)\frac{\partial}{\partial w_k}\wedge \frac{\partial}{\partial w_l}
\end{displaymath}
on $U\times V$.

\begin{thm}\label{K}
Assume that $(X,\pi)$ and $(Y,\sigma)$ are holomorphic Poisson manifolds.   If $X$ or $Y$ is compact, then there is an isomorphism
\begin{displaymath}
\bigoplus\limits_{p+q=k}H_{p}(X,\pi)\otimes_{\mathbb{C}} H_{q}(Y,\sigma)\cong H_{k}(X\times Y,\pi\oplus \sigma)
\end{displaymath}
for any $k$.
\end{thm}
\begin{proof}
Set $n=\textrm{dim}_{\mathbb{C}}X$, $m=\textrm{dim}_{\mathbb{C}}Y$ and let $pr_1$, $pr_2$ be the projections from $X\times Y$ onto $X$, $Y$, respectively.

Suppose that $\alpha\in \mathcal{A}^{\bullet}(X)$ and $\beta\in \mathcal{A}^{\bullet}(Y)$.
By the definition of $\pi\oplus\sigma$, we get
\begin{displaymath}
\begin{aligned}
l_{\pi\oplus\sigma}[pr_1^*(\alpha)\wedge pr_2^*(\beta)]=&l_{\pi\oplus 0}pr_1^*(\alpha)\wedge pr_2^*(\beta)+pr_1^*(\alpha)\wedge l_{0\oplus\sigma}pr_2^*(\beta)\\
 =&pr_1^*(l_{\pi}\alpha)\wedge pr_2^*(\beta)+pr_1^*(\alpha)\wedge pr_2^*(l_{\sigma}\beta),
\end{aligned}
\end{displaymath}
where the second equality uses the fact that $pr_{1*}(\pi\oplus 0)=\pi$ and $pr_{2*}(0\oplus\sigma)=\sigma$.
Hence,

\begin{equation}\label{external}
\partial_{\pi\oplus\sigma}[pr_1^*(\alpha)\wedge pr_2^*(\beta)]=pr_1^*(\partial_{\pi}\alpha)\wedge pr_2^*(\beta)+(-1)^{deg \alpha}pr_1^*(\alpha)\wedge pr_2^*(\partial_{\sigma}\beta).
\end{equation}

Consider the two double complexes
\begin{displaymath}
S^{\bullet,\bullet}:=ss\left(K^{\bullet,\bullet}(X,\pi)\otimes_{\mathbb{C}} K^{\bullet,\bullet}(Y,\sigma)\right) \quad\mbox{ and  } \quad T^{\bullet,\bullet}:=K^{\bullet,\bullet}(X\times Y,\pi\oplus\sigma).
\end{displaymath}
%\begin{displaymath}
%S^{\bullet,\bullet}=ss(K^{\bullet,\bullet}(X,\pi)\otimes_{\mathbb{C}} K^{\bullet,\bullet}(Y,\sigma)),
%\end{displaymath}
%and
%\begin{displaymath}
%T^{\bullet,\bullet}=K^{\bullet,\bullet}(X\times Y,\pi\oplus\sigma).
%\end{displaymath}
%The first pages of the spectral sequences associated to $S^{\bullet,\bullet}$ and $T^{\bullet,\bullet}$ are calculated as follows
By \cite[Proposition  2.7 (2)]{M2} and  (\ref{ss1}),
we have the first pages
\begin{displaymath}
\begin{aligned}
_{S}E_1^{p,q}
=&\bigoplus\limits_{\substack{a+b=p\\r+s=q}}H^{r}(K^{a,\bullet}(X,\pi))\otimes_{\mathbb{C}} H^{s}(K^{b,\bullet}(Y,\sigma))\\
=&\bigoplus\limits_{\substack{a+b=p\\r+s=q}}H^{-a,r}(X)\otimes_{\mathbb{C}} H^{-b,s}(Y)
\end{aligned}
\end{displaymath}
and $_{T}E_1^{p,q}=H^{-p,q}(X\times Y)$
of the spectral sequences associated to $S^{\bullet,\bullet}$ and $T^{\bullet,\bullet}$ respectively.
By (\ref{external}), $f=pr_1^*(\bullet)\wedge pr_2^*(\bullet):S^{\bullet,\bullet}\rightarrow T^{\bullet,\bullet}$ is a morphism between double complexes.
The morphism $E^{p,q}_1(f): _{S}E_1^{p,q}\rightarrow _{T}E_1^{p,q}$ at $E_1$-pages induced by $f$ is just the cartesian product
\begin{displaymath}
\bigoplus\limits_{\substack{a+b=p\\r+s=q}}H^{-a,r}(X)\otimes_{\mathbb{C}} H^{-b,s}(Y)\rightarrow H^{-p,q}(X\times Y).
\end{displaymath}
Notice that  $\Omega_{X\times Y}^{-p}=\bigoplus_{a+b=p}\left(\Omega_{X}^{-a}\boxtimes\Omega_{Y}^{-b}\right)$, where $\boxtimes$ means the analytic external tensor product of coherent analytic sheaves.
By \cite[IX, (5.23) (5.24)]{Dem}, $E^{p,q}_1(f)$ is an isomorphism for any $p$, $q\in\mathbb{Z}$,
so is the morphism $H^{k-n-m}(sS^{\bullet,\bullet})\rightarrow H^{k-n-m}(sT^{\bullet,\bullet})$ induced by $f$ for any $k\in\mathbb{Z}$.
By \cite[Proposition 2.7 (1)]{M2} and (\ref{computation1}),
\begin{displaymath}
\begin{aligned}
H^{k-n-m}(sS^{\bullet,\bullet})\cong&\bigoplus\limits_{p+q=k-n-m}H^{p}(K^{\bullet}(X,\pi))\otimes_{\mathbb{C}} H^{q}(K^{\bullet}(Y,\sigma))\\
=& \bigoplus\limits_{p+q=k}H_{p}(X,\pi)\otimes_{\mathbb{C}} H_{q}(Y,\sigma)
\end{aligned}
\end{displaymath}
and $H^{k-n-m}(sT^{\bullet,\bullet})\cong H_{k}(X\times Y,\pi\oplus\sigma)$.
We conclude this  theorem.
\end{proof}

\subsection{Blow-up formulae}
Suppose that $\rho:X\rightarrow Y$ is a proper  holomorphic map of complex manifolds.
We have the \emph{projection formula}
\begin{displaymath}\label{projection formula}
\rho_*(T\wedge \rho^*u)=\rho_*T\wedge u
\end{displaymath}
for any  $T\in \mathcal{D}^{\bullet,\bullet}(X)$ and $u\in \mathcal{A}^{\bullet,\bullet}(Y)$.
Hence
\begin{equation}\label{projection formula1}
\rho_*(\varphi\cup \rho^*\eta)=\rho_*\varphi\cup\eta
\end{equation}
for $\varphi\in H^{\bullet,\bullet}(X)$ and $\eta\in H^{\bullet,\bullet}(Y)$.
Let $\rho:(X,\pi)\rightarrow (Y,\sigma)$ be a proper surjective  morphism of holomorphic Poisson manifolds with the same dimensions.
Then $\rho_*(1_X)=\deg\rho\cdot 1_Y$, where $1_X$, $1_Y$ mean the currents defined by the constant $1$ on $X$, $Y$ respectively and $\deg \rho$ denotes the degree of $\rho$.
So
\begin{equation}\label{projection formula2}
\rho_*\rho^*=\deg \rho\cdot id:H_{k}(Y,\sigma)\rightarrow H_{k}(Y,\sigma).
\end{equation}
See also \cite[Theorem 3.6]{CCYY} for a Poisson modification $\rho$,
where a Poisson modification means a surjective morphism $\rho: (X,\pi)\rightarrow (Y,\sigma) $ of
compact holomorphic Poisson manifolds  with the same dimensions
satisfying that  there is an analytic subset $S\subseteq Y$ of codimension $\geq 2$ such that the restriction
$\rho: X-\rho^{-1}(S) \rightarrow  Y-S$ is biholomorphic.

Let $\rho:\widetilde{X}\rightarrow X$ be the blow-up of a  complex manifold $X$ along a  complex submanifold $Y$ with the exceptional divisor $E$.
Assume that $i_Y:Y\rightarrow X$ is the inclusion.
As we know, $\rho|_E:E\rightarrow Y$ can be naturally viewed as the projective bundle $\mathbb{P}(N_{Y/X})$ associated to the normal bundle $N_{Y/X}$ of $Y$ in $X$.
Let $t\in \mathcal{A}^{1,1}(E)$ be a first Chern form of the universal line bundle $\mathcal{O}_{E}(-1)$ on $E\cong{\mathbb{P}(N_{Y/X})}$ and let $h$ be the Dolbeault cohomology class of $t$.
Suppose that $i_E:E\rightarrow \widetilde{X}$ is the inclusion and $r=\textrm{codim}_{\mathbb{C}}Y\geq 2$.
Notice that
\begin{equation}\label{computation}
(\rho|_E)_*h^{i}=0,\quad 0\leq i\leq r-2, \qquad (\rho|_E)_*h^{r-1}=(-1)^{r-1},
\end{equation}
see \cite[p.20]{M3}.

\begin{lem}\label{blowup-Dolbeault}
For any $p$, $q\in\mathbb{Z}$,
\begin{displaymath}
\begin{aligned}
F:H^{p,q}(X)\oplus H^{p-1,q-1}(E)&\rightarrow H^{p,q}(\widetilde{X})\oplus H^{p-r,q-r}(Y)\\
(\alpha,\,\beta)^T& \mapsto (\rho^*\alpha+i_{E*}\beta,\,(\rho|_E)_*\beta)^T
\end{aligned}
\end{displaymath}
and
\begin{displaymath}
\begin{aligned}
G:H^{p,q}(\widetilde{X})\oplus H^{p,q}(Y)&\rightarrow H^{p,q}(X)\oplus H^{p,q}(E)\\
(\alpha,\,\beta)^T& \mapsto (\rho_*\alpha,\,i_E^*\alpha-(\rho|_E)^*\beta)^T
\end{aligned}
\end{displaymath}
are isomorphisms.
\end{lem}
\begin{proof}
Firstly, we prove that $F$ is an isomorphism.
\vspace{1mm}

Suppose that  $F(\alpha,\,\beta)=0$ for $\alpha\in H^{p,q}(X)$, $\beta\in H^{p-1,q-1}(E)$.
Then $\rho^*\alpha+i_{E*}\beta=0$ and $(\rho|_E)_*\beta=0$.
By \cite[Corollary 3.2]{M1},
$\beta=\sum_{i=0}^{r-1}h^{i}\cup (\rho|_{E})^*\theta_i$
for some $\theta_i\in H^{p-i-1,q-i-1}(Y)$, $0\leq i\leq r-1$.
By (\ref{projection formula1}) and (\ref{computation}), $(\rho|_E)_*\beta=(-1)^{r-1}\theta_{r-1}$,
%\begin{displaymath}
%\begin{aligned}
%(\rho|_E)_*\beta=&\sum_{i=0}^{r-1}(\rho|_E)_*[h^{i}\cup (\rho|_{E})^*\theta_i]\\
%=& (-1)^{r-1}\theta_{r-1},
%\end{aligned}
%\end{displaymath}
which implies that   $\theta_{r-1}=0$.
Therefore, $\beta=\sum_{i=0}^{r-2}h^{i}\cup (\rho|_{E})^*\theta_i$, and then
\begin{displaymath}
\rho^*\alpha+\sum_{i=1}^{r-1}i_{E*}[h^{i-1}\cup (\rho|_{E})^*\theta_{i-1}]=0.
\end{displaymath}
By \cite[Theorem 1.2]{M1}, $\alpha=0$ and $\theta_i=0$ for $0\leq i\leq r-2$.
So $\beta=0$.
Thus, $F$ is injective.
Give any $(\eta,\omega)\in H^{p,q}(\widetilde{X})\oplus H^{p-r,q-r}(Y)$.
Then $(-1)^{r-1}i_{E*}[h^{r-1}\cup (\rho|_E)^*\omega]-\rho^*i_{Y*}\omega\in H^{p,q}(\widetilde{X})$.
By \cite[Theorem 1.2]{M1}, there exist $\zeta\in H^{p,q}(X)$ and $\xi_i\in H^{p-i,q-i}(Y)$, $1\leq i\leq r-1$ such that
\begin{equation}\label{a}
(-1)^{r-1}i_{E*}[h^{r-1}\cup (\rho|_E)^*\omega]-\rho^*i_{Y*}\omega=\rho^*\zeta+\sum_{i=1}^{r-1}i_{E*}[h^{i-1}\cup (\rho|_{E})^*\xi_{i}].
\end{equation}
Pushforward (\ref{a}) by $\rho_*$, we have $\zeta=0$  by (\ref{projection formula1}).  So
\begin{equation}\label{b}
(-1)^{r-1}i_{E*}[h^{r-1}\cup (\rho|_E)^*\omega]-\rho^*i_{Y*}\omega=\sum_{i=1}^{r-1}i_{E*}[h^{i-1}\cup (\rho|_{E})^*\xi_{i}].
\end{equation}
By \cite[Theorem 1.2]{M1}, there exist $\gamma\in H^{p,q}(X)$ and $\delta_i\in H^{p-i,q-i}(Y)$ , $1\leq i\leq r-1$ such that
\begin{equation}\label{c}
\eta=\rho^*\gamma+\sum_{i=1}^{r-1}i_{E*}[h^{i-1}\cup (\rho|_{E})^*\delta_{i}].
\end{equation}
Set $\alpha:=\gamma-i_{Y*}\omega$ and $\beta:=\sum_{i=0}^{r-2}h^{i}\cup (\rho|_{E})^*(\delta_{i+1}-\xi_{i+1})+(-1)^{r-1}h^{r-1}\cup (\rho|_E)^*\omega$.
%\begin{displaymath}
%\beta:=\sum_{i=0}^{r-2}h^{i}\cup (\rho|_{E})^*(\delta_{i+1}-\xi_{i+1})+(-1)^{r-1}h^{r-1}\cup (\rho|_E)^*\omega.
%\end{displaymath}
%Then
%\begin{displaymath}
%\begin{aligned}
%\rho^*\alpha+i_{E*}\beta= &\rho^*\gamma-\rho^*i_{Y*}\omega+\sum_{i=0}^{r-2}i_{E*}[h^{i}\cup (\rho|_{E})^*(\delta_{i+1}-\xi_{i+1})]+(-1)^{r-1}i_{E*}[h^{r-1}\cup %(\rho|_E)^*\omega]\\
%= &\rho^*\gamma+\sum_{i=1}^{r-1}i_{E*}[h^{i-1}\cup (\rho|_{E})^*\xi_{i}]+\sum_{i=0}^{r-2}i_{E*}[h^{i}\cup (\rho|_{E})^*(\delta_{i+1}-\xi_{i+1})]\qquad(\mbox{by (\ref{b})})\\
%= &\eta \quad\qquad\qquad\qquad\mbox{ }(\mbox{by (\ref{c})}),
%\end{aligned}
%\end{displaymath}
We easily check that, $\rho^*\alpha+i_{E*}\beta=\eta$ by (\ref{b}), (\ref{c}) and $(\rho|_E)_*\beta=\omega$ by (\ref{projection formula1}), (\ref{computation}).
Hence $F$ is surjective.
Up to now, we proved that $F$ is an isomorphism.
\vspace{3mm}

Secondly, we prove that $G$ is an isomorphism.
\vspace{1mm}

Assume that  $Y$ is a Stein manifold.
Suppose that  $G(\alpha,\,\beta)^T=0$ for $\alpha\in H^{p,q}(\widetilde{X})$, $\beta\in H^{p,q}(Y)$.
Then $\rho_*\alpha=0$ and $i_E^*\alpha-(\rho|_E)^*\beta=0$.
By \cite[Theorem 1.2]{M1},
%$\alpha=\rho^*\gamma+\sum_{i=1}^{r-1}i_{E*} \left(h^{i-1}\cup (\rho|_{E})^*\beta_i\right)$.
\begin{displaymath}
\alpha=\rho^*\gamma+\sum\limits_{i=1}^{r-1}i_{E*} [h^{i-1}\cup (\rho|_{E})^*\beta_i]
\end{displaymath}
for some $\gamma \in H^{p,q}(X)$, $\beta_i\in H^{p-i,q-i}(Y)$, $1\leq i\leq r-1$.
Then
\begin{displaymath}
\begin{aligned}
(\rho|_E)^*\beta=&i_E^*\rho^*\gamma+\sum\limits_{i=1}^{r-1}i_E^*i_{E*}[h^{i-1}\cup (\rho|_{E})^*\beta_i]\\
=& (\rho|_E)^*i_Y^*\gamma+\sum\limits_{i=1}^{r-1} h^{i}\cup (\rho|_{E})^*\beta_i,
\end{aligned}
\end{displaymath}
where the second equality used  \cite[Lemma 4.4]{M1}.
By \cite[Corollary 3.2]{M1}, $\beta=i_Y^*\gamma$ and $\beta_i=0$ for $1\leq i\leq r-1$.
So $\alpha=\rho^*\gamma$.
By (\ref{projection formula1}), $\gamma=\rho_*\rho^*\gamma=\rho_*\alpha=0$ and then, $\alpha=0$, $\beta=0$.
Thus, $F$ is injective.
For any $(\eta,\omega)^T\in H^{p,q}(X)\oplus H^{p,q}(E)$,  $\omega=\sum_{i=0}^{r-1}h^{i}\cup (\rho|_{E})^*\theta_i$ for some $\theta_i\in H^{p-i,q-i}(Y)$   by \cite[Corollary 3.2]{M1}.
Set $\alpha:=\rho^*\eta+\sum_{i=1}^{r-1}i_{E*}(h^{i-1}\cup (\rho|_{E})^*\theta_{i})$ and $\beta:=i_Y^*\eta-\theta_0$.
By (\ref{projection formula1}) and (\ref{computation}),
\begin{displaymath}
\rho_*\alpha=\rho_*\rho^*\eta+\sum\limits_{i=1}^{r-1}i_{Y*}[(\rho|_{E})_*(h^{i-1}\cup (\rho|_{E})^*\theta_{i})]=\eta,
\end{displaymath}
%\begin{displaymath}
%\begin{aligned}
%\rho_*\alpha=&\rho_*\rho^*\eta+\sum\limits_{i=1}^{r-1}i_{Y*}[(\rho|_{E})_*(h^{i-1}\cup (\rho|_{E})^*\theta_{i})]\\
%=& \eta+\sum\limits_{i=1}^{r-1}i_{Y*}[(\rho|_{E})_*(h^{i-1})\cup \theta_{i}]=\eta,
%\end{aligned}
%\end{displaymath}
and by \cite[Lemma 4.4]{M1},
\begin{displaymath}
\begin{aligned}
i_E^*\alpha-(\rho|_E)^*\beta=&\left[i_E^*\rho^*\eta+\sum\limits_{i=1}^{r-1}i_E^*i_{E*}[h^{i-1}\cup (\rho|_{E})^*\theta_{i}]\right]-\left[(\rho|_{E})^*i_Y^*\eta-(\rho|_E)^*\theta_0\right]\\
=& \sum\limits_{i=0}^{r-1}h^{i}\cup (\rho|_{E})_*\theta_{i}=\omega.
\end{aligned}
\end{displaymath}
Thus, $G(\alpha,\,\beta)^T=(\eta,\,\omega)^T$.
Hence $G$ is surjective.
We proved that $G$ is an isomorphism if  $Y$ is Stein.

Go back to general cases.
Set $\widetilde{U}:=\rho^{-1}(U)$ and
\begin{displaymath}
\mathcal{F}^{p,q}(U):=\mathcal{A}^{p,q}(\widetilde{U})\oplus \mathcal{A}^{p,q}(Y\cap U),
\end{displaymath}
\begin{displaymath}
\mathcal{G}^{p,q}(U):=\mathcal{D}^{p,q}(U)\oplus \mathcal{A}^{p,q}(E\cap\widetilde{U}),
\end{displaymath}
\begin{displaymath}
g^{p,q}_U:\mathcal{F}^{p,q}(U)\rightarrow\mathcal{G}^{p,q}(U),\quad  (\alpha,\,\beta)\mapsto (\rho_*\alpha,\,i_E^*\alpha-(\rho|_E)^*\beta)
\end{displaymath}
for any open set $U\subseteq X$ and any $p$, $q\in\mathbb{Z}$.
Then $g^{\bullet,\bullet}_U$ gives a morphism  $(\mathcal{F}^{\bullet,\bullet}(U),\partial,\bar{\partial})\rightarrow (\mathcal{G}^{\bullet,\bullet}(U),\partial,\bar{\partial})$ of double complexes, and furthermore induces a morphism
\begin{displaymath}
G^{p,q}_U:H^{p,q}(\widetilde{U})\oplus H^{p,q}(Y\cap U)\rightarrow H^{p,q}(U)\oplus H^{p,q}(E\cap \widetilde{U}).
\end{displaymath}
Denote by $\mathcal{P}(U)$ the statement that $G^{p,q}_U$ are isomorphisms for all $p$, $q\in\mathbb{Z}$.
This lemma is equivalent to say that $\mathcal{P}(X)$ holds.

Now, we check that $\mathcal{P}$ satisfies the three conditions in \cite[Lemma 2.1]{M1}.
Obviously, $\mathcal{P}$ satisfies the disjoint condition.
For open sets $V\subseteq U$, denote by $\iota^U_V:\mathcal{F}^{p,q}(U)\rightarrow \mathcal{F}^{p,q}(V)$ and $j^U_V:\mathcal{G}^{p,q}(U)\rightarrow \mathcal{G}^{p,q}(V)$ the  corresponding restrictions.
Fix an integer $p$.
For open subsets $U$, $V$ in $X$, there is a commutative diagram
\begin{equation}\label{M-V-communication}
\xymatrix{
0\ar[r]&\mathcal{F}^{p,\bullet}(U\cup V)\ar[d]^{g^{p,\bullet}_{U\cup V}}\quad \ar[r]^{(\iota_U^{U\cup V},\iota_V^{U\cup V})^T\quad\quad} & \quad\mathcal{F}^{p,\bullet}(U)\oplus \mathcal{F}^{p,\bullet}(V)\ar[d]^{g^{p,\bullet}_{U}\oplus g^{p,\bullet}_{V}}\quad\ar[r]^{\quad\quad(\iota_{U\cap V}^U,-\iota_{U\cap V}^V)}& \quad\mathcal{F}^{p,\bullet}(U\cap V) \ar[d]^{g^{p,\bullet}_{U\cap V}}\ar[r]& 0\\
0\ar[r]&\mathcal{G}^{p,\bullet}(U\cup V)     \quad  \ar[r]^{(j_U^{U\cup V},j_V^{U\cup V})^T\quad\quad}& \quad\mathcal{G}^{p,\bullet}(U)\oplus \mathcal{G}^{p,\bullet}(V)    \quad\ar[r]^{\quad\quad (j_{U\cap V}^U,-j_{U\cap V}^ V)} & \quad\mathcal{G}^{p,\bullet}(U\cap V)     \ar[r]& 0 }
\end{equation}
 of  complexes.
By the exactness of (\ref{exact}),  the two rows in (\ref{M-V-communication}) are both exact sequences of complexes.
For convenience, set
\begin{displaymath}
I^{p,q}(U):=H^{p,q}(\widetilde{U})\oplus H^{p,q}(Y\cap U), \quad J^{p,q}(U):=H^{p,q}(U)\oplus H^{p,q}(E\cap \widetilde{U}).
\end{displaymath}
Then (\ref{M-V-communication}) induces a commutative diagram
\begin{displaymath}
\tiny{\xymatrix{
    \cdots\ar[r]&I^{p,q-1}(U\cap V)\ar[d]^{G^{p,q-1}_{U\cap V}}\ar[r]&I^{p,q}(U\cup V) \ar[d]^{G^{p,q}_{U\cup V}} \ar[r]& I^{p,q}(U)\oplus I^{p,q}(V) \ar[d]^{G^{p,q}_U\oplus G^{p,q}_V}\ar[r]&  I^{p,q}(U\cap V)\ar[d]^{G^{p,q}_{U\cap V}}\ar[r]&I^{p,q+1}(U\cup V)\ar[d]^{G^{p,q+1}_{U\cup V}}\ar[r]&\cdots\\
 \cdots\ar[r]&J^{p,q-1}(U\cap V)     \ar[r] & J^{p,q}(U\cup V)\ar[r]&J^{p,q}(U)\oplus J^{p,q}(V)       \ar[r]& J^{p,q}(U\cap V)     \ar[r] & J^{p,q+1}(\widetilde{U}\cup \widetilde{V})\ar[r]&\cdots}}
\end{displaymath}
of long exact sequences. If $G^{p,q}_U$, $G^{p,q}_V$ and $G^{p,q}_{U\cap V}$ are isomorphisms for all $p$, $q\in\mathbb{Z}$,
then so are $G^{p,q}_{U\cup V}$ for all $p$, $q\in\mathbb{Z}$ by the five-lemma.
Thus $\mathcal{P}$  satisfies the Mayer-Vietoris condition.
Let $\mathcal{U}$ be a basis of the topology of $X$ such that each $U\in \mathcal{U}$ is Stein.
Then  $Y\cap \bigcap\limits_{i=1}^lU_i$ is empty or Stein for any $U_1$, $\ldots$, $U_l\in\mathcal{U}$.
As we have proved, $G^{p,q}_{U_1\cap\ldots\cap U_l}$ is an isomorphism for every $p$, $q\in\mathbb{Z}$, so $\mathcal{P}$ satisfies the local condition.
By  \cite[Lemma 2.1]{M1}, $\mathcal{P}(X)$ holds, i.e., $G$ is an isomorphism.
\end{proof}

Suppose that $(Y,\pi|_Y)$ is a closed holomorphic Poisson submanifold of a holomorphic Poisson manifold $(X,\pi)$.
In this case, the conormal bundle $N^*_{Y/X}$ has a fiberwise Lie algebra structure given by the Poisson bracket.
Let  $\rho:\widetilde{X}\rightarrow X$ be the blow-up of  $X$ along  $Y$ and let $E$, $r$, $i_Y$ and $i_E$ be the ones defined as above.
We add the following assumption:
\begin{displaymath}
\mbox{($\star$)\qquad\emph{ $N^*_{Y/X,y}$ is an abelian Lie algebra over each $y\in Y$}.}\qquad\qquad\qquad\qquad\qquad\qquad\qquad
\end{displaymath}
By \cite[Propositions 8.2, 8.3]{P} or \cite[Proposition 3.15]{BCv}, there exists a unique holomorphic Poisson structure $\tilde{\pi}$ on $\widetilde{X}$ such that $\tilde{\pi}|_E$ is a holomorphic Poisson structure on $E$ and
$\rho:(\widetilde{X},\tilde{\pi})\rightarrow(X,\pi)$, $i_E:(E,\tilde{\pi}|_E)\rightarrow (\widetilde{X},\tilde{\pi})$,
$\rho|_E:(E,\tilde{\pi}|_E)\rightarrow (Y,\tilde{\pi}|_Y)$ are all morphisms between holomorphic Poisson manifolds.

Under the assumption $(\star)$, we have the following blow-up formulae.
\begin{thm}\label{blow-up-hKB}
Fix an integer $k$. Then
\begin{equation}\label{blow-up-hKB1}
(1)\qquad
\left(
  \begin{array}{cc}
  \rho^* &     i_{E*}  \\
     0   & (\rho|_E)_* \\
  \end{array}
\right)
: H_k(X,\pi)\oplus H_{k-1}(E,\tilde{\pi}|_E)\rightarrow H_{k}(\widetilde{X},\tilde{\pi})\oplus H_{k-r}(Y,\pi|_Y)
\end{equation}
is an  isomorphism, and moreover,
\begin{equation}\label{split-exact1}
\xymatrix{
0\ar[r]^{} &H_{k-1}(E,\tilde{\pi}|_E)\ar[r]^{(i_{E*},\,(\rho|_E)_*)^T\,\,\qquad} &
H_{k}(\widetilde{X},\tilde{\pi})\oplus H_{k-r}(Y,\pi|_Y)\ar[r]^{\quad\qquad (\rho_*,\,-i_{Y*})}&
H_k(X,\pi)\ar[r]^{} & 0
}
\end{equation}
is a split exact sequence, where $(\rho^*,0)^T$ is  a right  inverse of $(\rho_*,\,-i_{Y*})$\emph{;}
\begin{equation}\label{blow-up-hKB2}
(2)\quad
\left(
  \begin{array}{cc}
   \rho_*  &     0        \\
   i_E^*   & -(\rho|_E)^* \\
  \end{array}
\right)
: H_{k}(\widetilde{X},\tilde{\pi})\oplus H_{k-r}(Y,\pi|_Y)\rightarrow H_k(X,\pi)\oplus H_{k-1}(E,\tilde{\pi}|_E)
\end{equation}
is an isomorphism, and moreover,
\begin{equation}\label{split-exact2}
\xymatrix{
0\ar[r]^{} &H_k(X,\pi)\ar[r]^{(\rho^*,\,i_Y^*)^T\quad\qquad} &
H_{k}(\widetilde{X},\tilde{\pi})\oplus H_{k-r}(Y,\pi|_Y)\ar[r]^{\qquad (i_E^*,\,-(\rho|_E)^*)}&
H_{k-1}(E,\tilde{\pi}|_E)\ar[r]^{} & 0
}
\end{equation}
is a split exact sequence, where $(\rho_*,0)$ is  a left  inverse of $(\rho^*,\,i_Y^*)^T$.
\end{thm}
\begin{proof}
$(1)$
Consider the double complexes
\begin{displaymath}
S^{\bullet,\bullet}:=K^{\bullet,\bullet}(X,\pi)\oplus K^{\bullet,\bullet}(E,\tilde{\pi}|_E)[1,-1],
\end{displaymath}
\begin{displaymath}
T^{\bullet,\bullet}:=P^{\bullet,\bullet}(\widetilde{X},\tilde{\pi})\oplus P^{\bullet,\bullet}(Y,\pi|_Y)[r,-r].
\end{displaymath}
By (\ref{ss1}), we get the first pages
\begin{displaymath}
_{S}E_1^{p,q}=H^{-p,q}(X)\oplus H^{-p,q}(E) \quad\mbox{ and }\quad _{T}E_1^{p,q}=H^{-p,q}(\widetilde{X})\oplus H^{-p,q}(Y),
\end{displaymath}
of the spectral sequences associated to $S^{\bullet,\bullet}$ and $T^{\bullet,\bullet}$ respectively.
Define $f:S^{\bullet,\bullet}\rightarrow T^{\bullet,\bullet}$ as $(\alpha,\,\beta)\mapsto (\rho^*\alpha+i_{E*}\beta,\,(\rho|_E)_*\beta)$.
It is easy to check that $f$ is a morphism of double complexes.
By  Lemma \ref{blowup-Dolbeault}, $f$ induces an isomorphism  $_SE_1^{\bullet,\bullet}\rightarrow\mbox{ }  _TE_1^{\bullet,\bullet}$  at $E_1$-pages, hence induces an isomorphism $H^{k-n}(sS^{\bullet,\bullet})\tilde{\rightarrow} H^{k-n}(sT^{\bullet,\bullet})$ for any $k\in\mathbb{Z}$, where $n=\textrm{dim}_{\mathbb{C}}X$.
By (\ref{computation1}),
\begin{displaymath}
\begin{aligned}
H^{k-n}(sS^{\bullet,\bullet})=&H^{k-n}(K^{\bullet}(X,\pi))\oplus H^{k-n}(K^{\bullet}(E,\tilde{\pi}|_E))\\
\cong& H_{k}(X,\pi)\oplus H_{k-1}(E,\tilde{\pi}|_E),
\end{aligned}
\end{displaymath}
\begin{displaymath}
\begin{aligned}
H^{k-n}(sT^{\bullet,\bullet})=&H^{k-n}(P^{\bullet}(\widetilde{X},\tilde{\pi}))\oplus H^{k-n}(P^{\bullet}(Y,\pi|_Y))\\
\cong& H_{k}(\widetilde{X},\tilde{\pi})\oplus H_{k-r}(Y,\pi|_Y).
\end{aligned}
\end{displaymath}
Thus, (\ref{blow-up-hKB1}) is an isomorphism.

Suppose that $(i_{E*},\,(\rho|_E)_*)^T(\gamma)=0$ for $\gamma\in H_{k-1}(E,\tilde{\pi}|_E)$. Then
\begin{displaymath}
\left(
  \begin{array}{cc}
  \rho^* &     i_{E*}  \\
     0   & (\rho|_E)_* \\
  \end{array}
\right)
\left(
  \begin{array}{c}
  0  \\
  \gamma \\
  \end{array}
\right)
=
\left(
  \begin{array}{c}
  0  \\
  0 \\
  \end{array}
\right),
\end{displaymath}
which implies that $\gamma=0$ by (\ref{blow-up-hKB1}). So $(i_{E*},\,(\rho|_E)_*)^T$ is injective.
Evidently, $(\rho_*,\,-i_{Y*})\circ(i_{E*},\,(\rho|_E)_*)^T=0$.
Assume that $(\rho_*,\,-i_{Y*})(\alpha,\,\beta)^T=0$ for $\alpha\in H_{k}(\widetilde{X},\tilde{\pi})$ and $\beta\in H_{k-r}(Y,\pi|_Y)$.
By (\ref{blow-up-hKB1}), there exist $\zeta\in  H_k(X,\pi)$ and  $\eta\in H_{k-1}(E,\tilde{\pi}|_E)$
such that $\alpha=\rho^*\zeta+i_{E*}\eta$ and $\beta=(\rho|_E)_*\eta$.
Then $\zeta=\rho_*\alpha-i_{Y*}\beta=0$ by (\ref{projection formula2}), which implies that $(\alpha,\,\beta)^T=(i_{E*},\,(\rho|_E)_*)^T(\eta)$.
Hence $\textrm{im}(i_{E*},\,(\rho|_E)_*)^T=\textrm{ker}(\rho_*,\,-i_{Y*})$.
For any $\theta\in H_k(X,\pi)$, $(\rho_*,\,-i_{Y*})(\rho^*\theta,0)^T=\theta$ by (\ref{projection formula2}).
Thus,  $(\rho_*,\,-i_{Y*})$ is surjective.
We proved that (\ref{split-exact1}) is exact.
By (\ref{projection formula2}), $(\rho_*,\,-i_{Y*})\circ(\rho^*,0)^T=id$, hence (\ref{split-exact1}) is split and $(\rho^*,0)^T$ is  a right  inverse of $(\rho_*,\,-i_{Y*})$.

$(2)$
Consider the morphism
\begin{displaymath}
\begin{aligned}
g:K^{\bullet,\bullet}(\widetilde{X},\tilde{\pi})\oplus K^{\bullet,\bullet}(Y,\pi|_Y)&\rightarrow P^{\bullet,\bullet}(X,\pi)\oplus K^{\bullet,\bullet}(E,\tilde{\pi}|_E)\\
(\alpha,\,\beta)& \mapsto (\rho_*\alpha,\,i_E^*\alpha-(\rho|_E)^*\beta)
\end{aligned}
\end{displaymath}
of double complexes.
By  Lemma \ref{blowup-Dolbeault}, we easily get the isomorphism (\ref{blow-up-hKB2}) with the similar proof of (1).

Suppose that $(\rho^*,\,i_Y^*)^T(\gamma)=0$ for $\gamma\in H_k(X,\pi)$.
Then $\rho^*\gamma=0$.
By (\ref{projection formula2}), $\gamma=\rho_*\rho^*\gamma=0$.
So  $(\rho^*,\,i_Y^*)^T$ is injective.
Clearly, $(i_E^*,\,-(\rho|_E)^*)\circ(\rho^*,\,i_Y^*)^T=0$.
Assume that $(i_E^*,\,-(\rho|_E)^*)(\alpha,\,\beta)^T=0$  for $\alpha\in H_{k}(\widetilde{X},\tilde{\pi})$ and $\beta\in H_{k-r}(Y,\pi|_Y)$.
Set $\gamma:=\rho_*\alpha$.
Then
\begin{displaymath}
\left(
  \begin{array}{cc}
   \rho_*  &     0        \\
   i_E^*   & -(\rho|_E)^* \\
  \end{array}
\right)
\left(
  \begin{array}{c}
   \alpha         \\
   \beta    \\
  \end{array}
\right)
=
\left(
  \begin{array}{c}
   \gamma         \\
   0    \\
  \end{array}
\right)
=
\left(
  \begin{array}{cc}
   \rho_*  &     0        \\
   i_E^*   & -(\rho|_E)^* \\
  \end{array}
\right)
\left(
  \begin{array}{c}
   \rho^*\gamma         \\
   i_Y^*\gamma    \\
  \end{array}
\right)
,
\end{displaymath}
where the first equality use (\ref{projection formula2}).
By (\ref{blow-up-hKB2}), $(\alpha,\,\beta)^T=(\rho^*\gamma,\,i_Y^*\gamma)^T=(\rho^*,\,i_Y^*)^T(\gamma)$.
Thus $\textrm{im}(\rho^*,\,i_Y^*)^T=\textrm{ker}(i_E^*,\,-(\rho|_E)^*)$.
Moreover, $(i_E^*,\,-(\rho|_E)^*)$ is surjective by (\ref{blow-up-hKB2}).
Up to present, we proved that (\ref{split-exact2}) is exact.
By (\ref{projection formula2}),   $(\rho_*,0)\circ(\rho^*,\,i_Y^*)^T=\rho_*\circ\rho^*=id$, so (\ref{split-exact2}) is split and $(\rho_*,0)$ is  a left  inverse of $(\rho^*,\,i_Y^*)^T$.
\end{proof}

\begin{rem}
Using Lemma \ref{blowup-Dolbeault} instead of (\ref{blow-up-hKB1}) and (\ref{blow-up-hKB2}), we can prove the following results through almost the same procedures:
\begin{displaymath}
(1)\quad\qquad
\xymatrix{
0\ar[r]^{} &H^{p-1,q-1}(E)\ar[r]^{(i_{E*},\,(\rho|_E)_*)^T\,\,\qquad} &
H^{p,q}(\widetilde{X})\oplus H^{p-r,q-r}(Y)\ar[r]^{\quad\qquad (\rho_*,\,-i_{Y*})}&
H^{p,q}(X)\ar[r]^{} & 0
}
\end{displaymath}
is a split exact sequence, where $(\rho^*,0)^T$ is  a right  inverse of $(\rho_*,\,-i_{Y*})$;
\begin{displaymath}
(2)\quad\qquad\qquad\qquad
\xymatrix{
0\ar[r]^{} &H^{p,q}(X)\ar[r]^{(\rho^*,\,i_Y^*)^T\qquad} &
H^{p,q}(\widetilde{X})\oplus H^{p,q}(Y)\ar[r]^{\qquad (i_E^*,\,-(\rho|_E)^*)}&
H^{p,q}(E)\ar[r]^{} & 0
}
\end{displaymath}
is a split exact sequence, where $(\rho_*,0)$ is  a left  inverse of $(\rho^*,\,i_Y^*)^T$.
\end{rem}

\begin{cor}\label{blowup-weak}
For any $k\in\mathbb{Z}$,

$(1)$
$(\rho|_E)_*:H_{k-1}(E,\tilde{\pi}|_E)\rightarrow H_{k-r}(Y,\pi|_Y)$ is surjective,

$(2)$
$(\rho|_E)^*:H_{k-r}(Y,\pi|_Y)\rightarrow H_{k-1}(E,\tilde{\pi}|_E)$ is injective,

$(3)$
$(\rho^*,\,i_{E*}): H_k(X,\pi)\oplus H_{k-1}(E,\tilde{\pi}|_E)\rightarrow H_{k}(\widetilde{X},\tilde{\pi})$ is surjective,

$(4)$
$(\rho_*,\,i_E^*)^T:H_{k}(\widetilde{X},\tilde{\pi})\rightarrow H_k(X,\pi)\oplus H_{k-1}(E,\tilde{\pi}|_E)$
is injective,

$(5)$
$(\rho_*,\,i_E^*)^T:H_{k}(\widetilde{X},\tilde{\pi})\rightarrow H_k(X,\pi)\oplus [H_{k-1}(E,\tilde{\pi}|_E)/(\rho|_E)^*H_{k-r}(Y,\pi|_Y)]$
is an isomorphism.
\end{cor}
\begin{proof}
Immediately, Theorem \ref{blow-up-hKB} (1) implies  $(3)$ and Theorem \ref{blow-up-hKB} (2) implies  $(4)$.

Let $\gamma$ be any element in $H_{k-r}(Y,\pi|_Y)$.
By Theorem \ref{blow-up-hKB} (1), there exist $\alpha\in H_k(X,\pi)$ and $\beta\in H_{k-1}(E,\tilde{\pi}|_E)$ such that
\begin{displaymath}
\left(
  \begin{array}{c}
       0        \\
   \gamma     \\
  \end{array}
\right)
=\left(
  \begin{array}{cc}
   \rho_*  &     i_{E*}        \\
   0   & (\rho|_E)_* \\
  \end{array}
\right)
\left(
  \begin{array}{c}
       \alpha        \\
   \beta \\
  \end{array}
\right),
\end{displaymath}
which implies that $\gamma=(\rho|_E)_*\beta$.
Thus, $(\rho|_E)_*$ is surjective, i.e., (1) follows.

Suppose that $(\rho|_E)^*(\beta)=0$ for  $\beta\in H_{k-r}(Y,\pi|_Y)$.
Then
\begin{displaymath}
\left(
  \begin{array}{cc}
   \rho_*  &     0        \\
   i_E^*   & -(\rho|_E)^* \\
  \end{array}
\right)
\left(
  \begin{array}{c}
       0        \\
   \beta \\
  \end{array}
\right)
=\left(
  \begin{array}{c}
       0        \\
   0     \\
  \end{array}
\right),
\end{displaymath}
By Theorem \ref{blow-up-hKB} (2), $\beta=0$. We proved $(2)$.

%Assume that $(\rho_*,\,i_E^*)^T(\alpha)=(0,\,0)^T$ for $\alpha\in H_{k}(\widetilde{X},\tilde{\pi})$, i.e., $\rho_*(\alpha)=0$ and $i_E^*(\alpha)=0$.
%Then
%\begin{displaymath}
%\left(
%  \begin{array}{cc}
%   \rho_*  &     0        \\
%   i_E^*   & -(\rho|_E)^* \\
%  \end{array}
%\right)
%\left(
%  \begin{array}{c}
%   \alpha         \\
%   0    \\
%  \end{array}
%\right)
%=
%\left(
%  \begin{array}{c}
%   0         \\
%   0    \\
%  \end{array}
%\right).
%\end{displaymath}
%By Theorem \ref{blow-up-hKB} (2), $\alpha=0$. So $(4)$ is true.

Notice that
\begin{displaymath}
\left(
  \begin{array}{cc}
   \rho_*  &     0        \\
   i_E^*   & -(\rho|_E)^* \\
  \end{array}
\right)
(0\oplus H_{k-r}(Y,\pi|_Y))= 0\oplus (\rho|_E)^*H_{k-r}(Y,\pi|_Y),
\end{displaymath}
which implies $(5)$ by Theorem \ref{blow-up-hKB} (2).
\end{proof}

\begin{rem}
For compact cases,  X. Chen, Y. Chen, S. Yang and X. Yang \cite[Theorem 1.1]{CCYY} first proved Corollary \ref{blowup-weak} (5)  by the relative Koszul-Brylinski homology and the finiteness of dimensions of holomorphic Koszul-Brylinski homologies.
\end{rem}

\begin{cor}
Let  $\rho:\widetilde{X}\rightarrow X$ be the blow-up of  $X$ along a single point set $\{x_0\}$.
Assume that $T^*_{X,x_0}$ is an abelian Lie algebra.
Then
\begin{displaymath}
H_{k}(\widetilde{X},\tilde{\pi})\cong\left\{
 \begin{array}{ll}
H_k(X,\pi)\oplus \mathbb{C}^{n-1},&~k= n\\
 &\\
 H_k(X,\pi),&~\emph{others},
 \end{array}
 \right.
\end{displaymath}
where $n=\emph{dim}_{\mathbb{C}}X$.
\end{cor}
\begin{proof}
Clearly, $\pi|_{\{x_0\}}=0$, $E=\mathbb{C}P^{n-1}$ and  $N_{\{x_0\}/X,x_0}=T_{X,x_0}$.
By (\ref{zero2}) and (\ref{flagmanifold}), we have
\begin{displaymath}
H_{k-n}(\{x_0\},0)=\left\{
 \begin{array}{ll}
\mathbb{C},&~k=n\\
 &\\
 0,&~\textrm{others}
 \end{array}
 \right.
\textrm{and}\quad
H_{k-1}(E,\tilde{\pi}|_E)=\left\{
 \begin{array}{ll}
\mathbb{C}^n,&~k=n\\
 &\\
 0,&~\textrm{others}.
 \end{array}
 \right.
\end{displaymath}
By  Corollary \ref{blowup-weak}  (2), $(\rho|_E)^*(H_{k-n}(\{x_0\},0))$ is a one-dimensional subspace of $H_{k-1}(E,\tilde{\pi}|_E)$.
Thus, the corollary follows by  Corollary \ref{blowup-weak} (5).
\end{proof}

\subsection*{Acknowledgements}
 The author would like to thank Prof. Xiangdong Yang for his useful discussion.
 The author is supported by the National Natural Science Foundation of China (Grant No. 12001500, 12071444), the Scientific and Technological Innovation Programs of Higher Education Institutions in Shanxi (Grant No. 2020L0290) and the Fundamental Research Program of Shanxi Province (Grant No. 201901D111141).
%==================================

%==================================
\subsection*{Data Availability}
 This article has no associated data.
%==================================

%==============================================================

\end{document}